\documentclass[11pt,thmsa]{article}%
\usepackage{amsmath}
\usepackage{amsfonts}
\usepackage{amssymb}
\usepackage{graphicx}%
\newtheorem{theorem}{Theorem}[section]

\newtheorem{corollary}[theorem]{Corollary}

\newtheorem{lemma}[theorem]{Lemma}

\newtheorem{proposition}[theorem]{Proposition}

\newenvironment{proof}[1][Proof]{\noindent\textbf{#1.} }{\ \rule{0.5em}{0.5em}}

\begin{document}
\title{The number of geometrically distinct reversible closed geodesics on a Finsler sphere with $K\equiv 1$\thanks{Supported by NSFC (no. 11771331) and NSF of Beijing (no. 1182006)}}
\author{Ming Xu\\
\\
School of Mathematical Sciences\\
Capital Normal University\\
Beijing 100048, P. R. China\\
Email:mgmgmgxu@163.com
\\
}
\date{}
\maketitle

\begin{abstract}
In this paper we study the Finsler sphere $(S^n,F)$ with $n>1$, which has constant flag
curvature $K\equiv 1$ and only finite prime closed geodesics. In this
case, the connected isometry group $I_0(S^n,F)$ must be
a torus which dimension satisfies $0<\dim I(S^n,F)
\leq[\frac{n+1}{2}]$. We will prove that the number of geometrically distinct reversible closed geodesics
on $(S^n,F)$ is at least $\dim I(S^n,F)$. When
$\dim I_0(S^n,F)=[\frac{n+1}{2}]$, the equality happens, and there are exactly $2[\frac{n+1}{2}]$ prime closed geodesics, which verifies Anosov conjecture in this special case.



\end{abstract}
\section{Introduction}
Finsler sphere with constant flag curvature $K\equiv 1$ is an important subject in Finsler geometry. In the Randers case, D. Bao, Z. Shen and C. Robles provided a complete classification \cite{BRS2004}. Meanwhile,
R. L. Bryant used contact geometry and complex geometry to construct a new family of Finsler metrics on spheres with $K\equiv 1$
\cite{Br1996}\cite{Br1997}\cite{Br2002}.
The navigation process with a Killing vector field can provide new examples from old ones
\cite{HM2015}. People guess there would be more, but new examples are
very hard to be found. On the other hand, R. L. Bryant, L. Huang and X. Mo have constructed many local examples with constant flag curvature \cite{BHM2017}.

The geodesics on a Finsler
sphere $S^n$ with $K\equiv 1$ could be much more complicated than in Riemannian geometry. Recently,
R. L. Bryant, P. Foulon, S. Ivanov, V. S. Matveev and W. Ziller purposed the idea of classifying Finsler spheres with $K\equiv 1$ up to conjugation, i.e. the behavior of geodesics. In particular, they provided a classification when $n=2$
\cite{BFIMZ2017}.

The key observation in \cite{BFIMZ2017} is the following. On a Finsler sphere $(S^n,F)$ with $n>1$ and $K\equiv 1$, there exists a unique Clifford Wolf translation $\psi$ moving each point with a distance $\pi$ \cite{Sh1997}. Because it is generalized from the antipodal map for a standard Riemannian sphere, we will simply call $\psi$ the {\it antipodal map}. Further more, if there exist only finite prime closed geodesics, (after taking the closure) we can use this $\psi$ to "generate" a continuous family of isometries.

In this paper, we will use this observation to estimate the number of prime closed geodesics on a Finsler sphere $S^n$ with $n>1$ and $K\equiv 1$. The usual tools for studying this problem are variational method, index theory and infinitely dimensional
Morse theory. There are many literatures in this field, for example \cite{Du2016}\cite{DL2009}\cite{Ra1989}\cite{Wa2015}. However, when the manifold has constant flag curvature, the above observation from \cite{BFIMZ2017} brings Lie theory to help us simply the steps and bring more results.

For example, the Anosov conjecture claims on a Finsler sphere $(S^n,F)$,
there exists at least $2[\frac{n+1}{2}]$ prime closed geodesics \cite{An1975}. It has been proven in some special cases, for
example, when $n$ is small or the metric is bumpy and satisfies certain pinch condition for its flag curvature \cite{BL2010}\cite{Wa2012}.

It is natural to ask if the Anosov conjecture is true for a
Finsler sphere $S^n$ with $K\equiv 1$. When $n=2$, the answer
is obviously yes. Moreover,
Theorem 2 in \cite{BFIMZ2017} indicates the number of prime
closed geodesics on a Finsler sphere $S^2$ with $K\equiv 1$ is
either 2 or infinity. When it is 2, the two prime closed geodesics
are the same curve with different directions.
On the other hand, when $n>2$, we do not know the answer. This is the key problem we will pursue in this paper. Notice that in this case, the pinch condition for the flag curvature is always satisfied, but we can not easily deduce the bumpy condition for each closed geodesic, i.e. we can not apply Theorem 1.2 in \cite{Wa2012} to this case. So we switch to Lie theory for help and borrow some techniques from \cite{BFIMZ2017} to show to what extent the Anosov conjecture
is correct on a Finsler sphere $S^n$ with $K\equiv 1$.

Assume $(M,F)=(S^n,F)$ with $n>1$ is a Finsler sphere with $K\equiv 1$ and only finite prime closed geodesics.
We will first prove some fundamental properties of its geometry. For example, each prime closed geodesic has no self intersection. Two closed geodesics intersect only when they are {\it geometrically} the same, i.e. they are the same subsets of $M$. But notice that they may have different multiples or directions. If two geodesics are geometrically the same but have different directions, we call them {\it reversible}. Most importantly, we have the following fundamental theorem for the
identity component $I_0(M,F)$ of the isometry group $I(M,F)$ (later we will simply call $I_0(M,F)$ the {\it connected isometry group}).

\begin{theorem}\label{main-thm-0}
Let $(M,F)=(S^n,F)$ with $n>1$ be a Finsler sphere with $K\equiv 1$ and only finitely many prime closed geodesics.
Then $I_0(M,F)$ is a torus which dimension satisfies $0<\dim I(M,F)\leq [\frac{n+1}{2}]$.
\end{theorem}

Then we can introduce our main theorems in this paper.

\begin{theorem}\label{main-thm-1}
Let $(M,F)=(S^n,F)$ with $n>1$ be a Finsler sphere with $K\equiv 1$ and only finite prime closed geodesics. Then the number of geometrically distinct reversible closed geodesics is at least $\dim I(M,F)$.
\end{theorem}

\begin{theorem}\label{main-thm-2}
Let $(M,F)=(S^n,F)$ with $n>1$ be a Finsler sphere with $K\equiv 1$ and only finite prime closed geodesics. If $\dim I_0(M,F)=[\frac{n+1}{2}]$, then each
closed geodesic on $M$ is reversible and the number of geometrically distinct closed geodesics is exactly
$\dim I(M,F)$.
\end{theorem}

Combining Theorem \ref{main-thm-0} and Theorem \ref{main-thm-1}, we can easily get the following corollary.

\begin{corollary}On a Finsler sphere with constant flag curvature and only finite prime closed geodesics, there exists at least one reversible closed geodesic.
\end{corollary}

Notice that each geometrically distinct reversible closed geodesic is counted as two. So
Theorem \ref{main-thm-2} indicates that the Anosov conjecture is correct for Finsler spheres with constant flag curvature
and maximal $\dim I(M,F)$, and it provides more precise information. The Katok metrics
with only finite prime closed geodesics belong to the case. Theorem \ref{main-thm-2} privides an intrinsic explanation for why the number of prime closed geodesics for those examples are exactly $[\frac{n+1}{2}]$.

The proof of
Theorem \ref{main-thm-1} relies on a fixed point technique. The fix point set $N\subset M=S^n$ for a family of isometries (or its special case, the common zero set for a family of Killing vector fields) is either a two-points $\psi$-orbit, or a connected totally geodesic
submanifold. So when $\dim N>0$, $(N,F|_N)$ is also a Finsler sphere with $K\equiv 1$ and only finite prime closed geodesic. Then we can use this observation to reduce the dimension and prove Theorem \ref{main-thm-1} by induction.

The proof of Theorem \ref{main-thm-2} when $M$ is an odd dimensional sphere applies an $S^1$-equivariant Morse theory. In this case, the abundance of Killing vector fields implies we can find a generic $X$
from them such that $X$ generates the action of $T=S^1$, and the geodesics are non-degenerated critical $S^1$-orbits for the Morse function $f(\cdot)=F(X(\cdot))^2$. Direct calculation shows each geodesic has an even index and a trivial negative bundle. By the standard argument for the Morse's Lacunary Principle (the $T$-equivariant version) \cite{Bo1954}\cite{Bo1982}\cite{Hi1984}, we see
the number of geometrically distinct prime closed geodesics coincide with
the $T$-equivariant Euler number of $M=S^n$,
which can be calculated by the Leray-Serre spectral sequence.

This paper is organized as following. In Section 2, we recall some fundamental knowledge on Finsler geometry. In Section 3,
we discuss some geometric properties of Finsler spheres with $K\equiv 1$ and only finite prime closed geodesics. In Section 4 and Section 5, we prove Theorem
\ref{main-thm-1} and Theorem \ref{main-thm-2} respectively.

{\bf Acknowledgement.} I would like to thank Wolfgang Ziller for sending me his most recent preprint which inspired this work, and also thank Wolfgang Ziller, Huagui Duan and Xuchao Yao for helpful discussion.

\section{Priliminaries}

\subsection{Finsler metric and Finsler space}
A {\it Finsler metric} on a smooth $n$-dimensional manifold $M$ is a
continuous function $F:TM\rightarrow [0,+\infty)$ satisfying the following conditions:
\begin{description}
\item{\rm (1)} $F$ is positive and smooth on the slit tangent
bundle $TM\backslash 0$.
\item{\rm (2)} For any $y\in T_xM$ and $\lambda>0$, $F(x,\lambda y)=\lambda F(x,y)$.
\item{\rm (3)} When $y$ is a nonzero vector in $T_xM$, the Hessian
matrix $$(g_{ij}(x,y))=\frac{1}{2}[F^2(x,y)]_{y^iy^j}$$
is positive definite, with respect to any {\it standard local coordinates} $x=(x^i)\in M$ and $y=y^j\partial_{x^j}\in T_xM$.
\end{description}
We call $(M,F)$ a {\it Finsler space} or {\it Finsler manifold}.

The Hessian matrix defines an inner product
$$\langle u,v\rangle^F_y=g_{ij}(x,y)u^iv^j=
\frac{1}{2}\frac{\partial^2}{\partial s\partial t}F^2(y+su+tv)|_{s=t=0},$$
which will also be denoted as $g_y^F$ sometimes.
When $g_y^F$ is independent of the choice of $y$ everywhere,
the metric $F$ is Riemannian. The most simple and the most important non-Riemannian Finsler metric is Randers metric,
which is of the form $F=\alpha+\beta$ where $\alpha$ is
a Riemannian metric and $\beta$ is a 1-form.

\subsection{Geodesic of constant speed}

The Finsler metric $F$ on $M$ defines
a distance function $d_F(\cdot,\cdot)$ which does not
satisfy the reversibility $d_F(x',x'')=d_F(x'',x')$ in general.
A smooth curve $c(t)$ on $M$ satisfying
locally minimizing principle is called a {\it geodesic}. In this work,
we usually choose the parameter along the geodesic such that it
has constant speed, i.e. $F(\dot{c}(t))\equiv\mathrm{const}>0$.
Then a geodesic with (positive) constant speed can be equivalently defined as the integration curve of the geodesic spray vector field on $TM\backslash 0$, i.e. $c(t)$ is a geodesic iff for any standard local coordinates, it satisfies
$$\ddot{c}^i(t)+\mathrm{G}^i(c(t),\dot{c}(t))=0,\quad\forall i,$$
where $\mathrm{G}^i=\frac{1}{4}g^{il}([F^2]_{x^jy^l}y^j-[F^2]_{x^l})$
for any standard local coordinates.

\subsection{Flag curvature}

There are many notions of curvature in Finsler geometry. In this paper,
we will only work with flag curvature, which is the natural generalization of sectional curvature in Riemannian geometry.
Let $x$ be a point on $(M,F)$, $y$ a nonzero vector in $T_xM$ (the flag pole), $\mathbf{P}=\mathrm{span}\{y,u\}\subset T_xM$ a tangent plane containing $y$ (the flag), then the flag curvature
for the triple $(x,y,\mathbf{P})$ is defined as
$$K^F(x,y,\mathbf{P})=\frac{\langle R^F_y u,u\rangle_y^F}{
\langle y,y\rangle_y^F\langle u,u\rangle_y^F-
[\langle y,u\rangle_y^F]^2},$$
where $R^F_y:T_xM\rightarrow T_xM$ is the Riemann curvature operator in the Jacobi equation for the variation of a family
of geodesics of constant speeds. See \cite{BCS2000} or \cite{Sh2001} for its explicit presenting by standard local coordinates.

When the metric $F$ is Riemannian, the flag curvature is
just the sectional curvature, which only depends on the point $x$ and the flag $\mathbf{P}$. The dependence of the flag curvature on the flag pole $y$ implies it is a much more localized concept
than the sectional curvature.

Z. Shen provided an important observation relating flag curvature to sectional curvature, which can be refined as following. Let $Y$ be a smooth vector field on $(M,F)$
which is nonvanishing everywhere. Then the inner products
$g^F_{Y(x)}$ defines a Riemannian metric on $M$, which is called the osculation metric of $F$ at $Y$. The following theorem (see Theorem 4.2 in \cite{XDHH2017})
relates the flag curvature $K^F(x,y,\mathbf{P})$ for $y=Y(x)$ and any tangent plane $\mathbf{P}$ containing $y$ and the sectional curvature
$K^{g_Y^F}(x,P)$.

\begin{theorem}\label{thm-geodesic-field}
Let $Y$ be a smooth vector field on $(M,F)$ which is non-vanishing everywhere. Assume the integration curve of
$Y$ passing $x$ is a geodesic of constant speed, then for $y=Y(x)$ and any tangent plane $\mathbf{P}$ containing
$y$, we have $K^{F}(x,y,\mathbf{P})=
K^{g_Y^F}(x,\mathbf{P})$.
\end{theorem}

In particular, if further more $Y$ is a geodesic field, i.e. each integration curve of $Y$ is a geodesic of constant speed, then Theorem \ref{thm-geodesic-field}
can be applied to all the points of $M$. That is the original
statement of Z. Shen's theorem in \cite{Sh2001}.

\subsection{Totally geodesic sub-manifold}

A sub-manifold $N$ of a Finsler manifold $(M,F)$ can be naturally endowed with a sub-manifold metric, denoted as $F|_N$. We call
$(N,F|_N)$ a {\it Finsler subspace} or a {\it Finsler sub-manifold}.

The Finsler sub-manifold $(N,F|_N)$ of $(M,F)$ is called {\it totally geodesic} if any geodesic of $(N,F|_N)$ is also a geodesic of $(M,F)$ \cite{XD2017}.

Notice that if $\dim N=0$, it is always totally geodesic. If $\dim N=1$, $N$ is totally geodesic iff each connected component is
a reversible geodesic. If $\dim N>1$, then we have the following easy but important property for $(N,F|_N)$.

\begin{proposition}
Let $N$ be a connected totally geodesic subspace of the Finsler space $(M,F)$, then for any $x\in N$, nonzero vector $y\in T_xN$, and tangent plan $\mathbf{P}\subset T_xN$ containing $y$, we have $K^F(x,y,\mathbf{P})=K^{F|_N}(x,y,\mathbf{P})$.
\end{proposition}

For example, the fixed point set of a family of isometries on
$(M,F)$ provides a special class of closed imbedded totally geodesic sub-manifold, in which there is a sub-class provided by the common zero set of a family of Killing vector field.

\subsection{Isometry and Killing vector field}

An {\it isometry} on a Finsler space $(M,F)$ is a diffeomorphism $\varphi:M\rightarrow M$ satisfying $\varphi^*F=F$. Its infinitesimal generator is a
{\it Killing vector field} $X$ defined by $L_XF=F$.
We denote $I(M,F)$ the group of all isometries of $(M,F)$,
and $I_0(M,F)$ the identity component of $I(M,F)$ (we call it the {\it connected isometry group}). They are Lie groups \cite{DH2002}, which
Lie algebra $\mathfrak{g}=\mathrm{Lie}(I(M,F))$ coincides with the space
of all Killing vector fields.

In later discussion, we will need the following Lemma, which follows Lemma 3.1 in \cite{DX2004} immediately.

\begin{lemma}\label{lemma-2-1}
Let $X$ be a Killing vector field of the Finsler
space $(M,F)$ which does not vanish at $x\in M$. Then $x$ is
a critical point for the function $f(\cdot)=F(X(\cdot))^2$ iff
the integration curve of $X$ passing $x$ is a geodesic.
\end{lemma}

\section{Finsler sphere with $K\equiv 1$ and only finite prime closed geodesics}

\subsection{The antipodal map and the exponential map}

We start by assuming that $(M,F)$ is a connected and simply connected Finsler space with $K\equiv 1$. Later we will see $M$ must be diffeomorphic to a standard sphere.

Firstly, we introduce the antipodal map $\psi$ in \cite{BFIMZ2017} and \cite{Sh1997}.
For any $x\in M$, let $c(t)$ be a unit speed geodesic (i.e. $F(\dot{c}(t))\equiv 1$) such that $x=c(0)$. Then the Jacobi equation indicates $x^*=c(\pi)$ is the first conjugation point.
Changing the geodesic $c(t)$ among all unit speed geodesics outgoing from $x$, the corresponding $x^*=c(\pi)$ must be fixed. So the map $\psi$ from $x$ to $x^*$ is well defined. When $(M,F)$ is a standard Riemannian sphere, $\psi$ is just the
antipodal map. So we will simply call $\psi$ the {\it antipodal map} and notice that $\psi^2\neq \mathrm{id}$ in general.

In \cite{BFIMZ2017}, it has been shown that $\psi$ is an isometry of $(M,F)$. Furthermore, we have the following easy lemma.

\begin{lemma}\label{lemma-2-0}
The antipodal map $\psi$ is a Clifford Wolf translation which
is contained in the center of $I(M,F)$.
\end{lemma}

Recall that $\varphi$ is a Clifford Wolf translation on $(M,F)$ iff it is an isometry satisfying
$d(x,\varphi(x))\equiv\mathrm{const}$ \cite{DX2004}. When the metric is not reversible, it may happen that $\varphi$ is a Clifford Wolf, but $\varphi^{-1}$ is not.

The exponential map $\mathrm{Exp}_x$ at any $x\in M$
is define as following. Let $c(t)$ be any constant speed geodesic with $c(0)=x$ and $\dot{c}(0)=u$, then $\mathrm{Exp}_x(u)=c(1)$. It is a
surjective map from the closed ball
$$B^n_o(\pi)=\{x|F(x)\leq\pi\}\subset T_xM$$
to $M$, which maps the boundary $S^{n-1}_o(\pi)\subset B^n_o(\pi)$ to $x^*$. As being pointed out in \cite{BFIMZ2017}, it induces a homoemorphism  from a standard $S^n$ to $M$.
In fact we can prove an even stronger statement, i.e.
\begin{lemma}\label{lemma-3-1}
Any connected and simply connected Finsler space $(M,F)$ with $K\equiv1$ is diffeomorphic to a standard sphere.
\end{lemma}

\begin{proof}
We just need to prove $M$ is $C^1$-homoemorphic to a standard sphere, then the $C^1$-homoemorphism can be perturbed to be
a smooth diffeomorphism.

We take the local charts on $M$ as following. Let $x$ be any point on $M$, $x^*=\psi(x)$, $\langle\cdot,\cdot\rangle$ be the inner product
on $T_xM$ and $||\cdot||$ the corresponding norm. Denote $U_1$ and $U_2$ two copies of
$$U=\{u\in T_xM\mbox{ with } ||u||<\frac{3\pi}{4}\},$$
and $V_1$ and $V_2$ two other open balls in $T_xM$, i.e.
\begin{eqnarray*}
V_1=\{u|F(u)<\frac{3\pi}4\}\mbox{ and }
V_2=\{u|F(-u)<\frac{3\pi}4\}.
\end{eqnarray*}

We can find two diffeomorphisms $\varphi_i:U_i\rightarrow V_i$, $i=1$, $2$, such that when $\frac{\pi}{4}<||u||<\frac{3\pi}{4}$,
\begin{equation*}
\varphi_1(u)=\frac{||u||}{F(u)}u\mbox{ and }
\varphi_2(u)=\frac{||u||}{F(-u)}u.
\end{equation*}

Define the two chart map $\phi_i:U_i\rightarrow M$, such that
\begin{eqnarray*}
\phi_1(u)=\mathrm{Exp}_x(\varphi_1(u)),
\mbox{ and }\\
\phi_2(u)={\mathrm{Exp}}^*_{x^*}(\psi_*(\varphi_2(u))),
\end{eqnarray*}
in which $\mathrm{Exp}^*_{x^*}$ is the exponential map of an incoming version at $x^*$, i.e. for any constant speed geodesic $c(t)$
with $c(0)=x^*$ and $\dot{c}(0)=-v$,
$\mathrm{Exp}^*_{x^*}(v)=c(-1)$.
Notice that $\mathrm{Exp}_x$ is a $C^1$-map around $o\in T_xM$ \cite{BCS2000}. Similarly so does $\mathrm{Exp}^*_{x^*}$.

By the definitions of the antipodal map and the exponential maps, it is not hard to check that
the transfer map between the two charts
$$\phi_{12}=\phi_2\circ\phi_1^{-1}:\{u\in U_1, \frac{\pi}{4}<||u||<\frac{3\pi}{4}\}
\rightarrow \{u\in U_2, \frac{\pi}{4}<||u||<\frac{3\pi}{4}\}$$
satisfies
$$\phi_{12}(u)=-\frac{\pi-||u||}{||u||}u.$$
So by the charts $(U_1,\phi_1)$, $(U_2,\phi_2)$, the identification between $U_1$ and $U_2$ and the translation map
$\phi_{12}$, $M$ is identified with a standard $S^n$, up to
a $C^1$-homeomorphism.
\end{proof}

\subsection{Totally geodesic sub-manifolds}

Then we turn to totally geodesic sub-manifolds of a Finsler sphere $(M,F)=(S^n,F)$ with $n>1$ and $K\equiv 1$.

By the same argument as in the proof of Lemma \ref{lemma-3-1}, we can prove

\begin{lemma}\label{lemma-3-2}
Let $(M,F)=(S^n,F)$ be a Finsler sphere with $K\equiv 1$, and $N$ a closed connected imbedded totally geodesic submanifold with $\dim N>1$. Then $(N,F|_N)$ is also a Finsler sphere with $K\equiv 1$.
\end{lemma}

When $N$ in Lemma \ref{lemma-3-2} is a hypersurface, i.e. of co-dimension one, we have the following lemma, which will be
used in Section 5.

\begin{lemma} \label{lemma-3-3}
Let $(M,F)=(S^n,F)$ be a Finsler sphere with $n>1$ and $K\equiv 1$. Let
 $N$ be a closed imbedded totally geodesic hypersurface of $(M,F)$.
Then we have the following:
\begin{description}
\item{\rm (1)} $N$ must be connected.
\item{\rm (2)} $M\backslash N$ has two
connected components.
\item{\rm (3)} The action of $\psi$ on the two connected components of $M\backslash N$ exchanges them.
\item{\rm (4)} Any geodesic of $M$ either belong to $N$ or intersects with $N$.
    \end{description}
\end{lemma}
\begin{proof}
Let $x$ be a point on $N$, $x^*=\psi(x)$ and $N'$ be the connected component of $N$
containing $x$. Then the exponential map $\mathrm{Exp}_x$ is a surjective map from
$(B^n_o(\pi),S^{n-1}_o(\pi))$ to $(M,x^*)$, and at the same time, it is surjective from $(B^n_o(\pi)\cap T_xN, S^{n-1}_o(\pi)\cap T_xN)$ to $(N',x^*)$. So $M\backslash N'$ has two connected components.

To see the action of $\psi$ on these two components, we only need to look at a geodesic $c(t)$ with $c(0)=x$ which is not
contained in $N'$. The open geodesic segment $c(t)$ with $t\in (0,\pi)$ is contained in one component of $M\backslash N'$. Its
$\psi$-image, the geodesic segment $c(t)$ with $t\in (\pi,2\pi)$, is contained in another connected component of $M\backslash N'$.
So the action of $\psi$ exchanges the two components of $M\backslash N'$. At the same time, we see any geodesic of $M$, if it is not contained in $N'$, must intersect with $N'$.

Finally, we prove that $N$ can not have other connected components. Assume conversely that $N$ has another component $N''$, then any geodesics in $N''$ has no intersection with $N'$. This is a contradiction because $N$ is imbedded.
\end{proof}

The fixed point set of a family of isometries (or more specially, the common zero set of a family of Killing vector fields) has more speciality than general totally geodesic sub-manifolds.
For example, the following lemma tells us in many cases
$N$ must be connected.

\begin{lemma}\label{lemma-3-4}
Let $(M,F)=(S^n,F)$ be a Finsler sphere with $n>1$ and $K\equiv1$, and $N$ the fixed point set of a family of isometries of $(M,F)$. Then $N$ must satisfy one of the following:
\begin{description}
\item{\rm (1)} $N$ is a two-points $\psi$-orbit, i.e. $N=\{x',x''\}$ with $d_F(x',x'')=d_F(x'',x')=\pi$.
\item{\rm (2)} $N$ is a reversible closed geodesic.
\item{\rm (3)} $(N,F|_N)$ is a Finsler sphere with $\dim N>1$ and $K\equiv 1$.
\end{description}
\end{lemma}
\begin{proof}
Whenever $N$ contains two points $x$ and $x'$ such that $d_F(x,x')\neq\pi$, there exists a unique shortest geodesic
$\gamma$ from $x$ to $x'$, which is also contained in the fixed point set $N$. Repeating this argument, we can see
$N$ must be connected whenever it contains more than two points. The cases (2) and (3) correspond to $\dim N=1$ and $\dim N>1$ respectively.
Finally we consider the case $\dim N=0$. By Lemma \ref{lemma-2-0}, for any $x\in N$,  $x^*=\psi(x)$ is another point in $N$. So $N$ contains at least two points. It can not contain the third, thus we must have $\psi(x^*)=x$.
\end{proof}

\subsection{The case when there exist only finite prime closed geodesics}

From now on, we further assume

{\bf Assumption (F).} $(M,F)$ has only finite prime closed geodesics.

The antipodal map $\psi$ can not be of
finite order, i.e. there does not exist an integer $k$ such
that $\psi^k=\mathrm{id}$, otherwise by the property of Clifford Wolf translation, or as argued in \cite{BFIMZ2017}, each orbit of $\psi$ is
contained in a closed geodesic, which is  a contradiction
to Assumption (F). The following lemma indicates there can not be too many finite $\psi$-orbits in $(M,F)$.

\begin{lemma}\label{lemma-3-5}
Let $(M,F)=(S^n,F)$ with $n>1$ be a Finsler sphere with $K\equiv 1$ and only finite prime closed geodesics.
If the union $N$ of all finite $\psi$-orbits is not empty, then it
must be one of the following:
\begin{description}
\item{\rm (1)} $N$ is a two-points $\psi$-orbit.
\item{\rm (2)} $N$ is a reversible closed geodesic.
\end{description}
Further more, in case (2), the length of $N$ in each direction is a rational multiple of $\pi$. If for either direction the length of $N$ is $l=p\pi/q$ with $p,q\in \mathbb{Z}$, then $N$ is the fixed point set of
    $\psi^q$.
\end{lemma}
\begin{proof} For each integer $q>1$,
denote $N_q$ the fixed point set of $\psi^q$. Each non-empty $N_{q}$ must satisfy
$\dim N_q<2$, otherwise there will be infinitely many closed geodesics. If there are some $N_q$ with $\dim N_q=1$, we
may choose $q$ to be the smallest, then we claim

 {\bf Claim 1}. $N_{q}=N_{q'}$ when
$q'$ is a multiple of $q$, and $N_{q'}$ is empty otherwise.

Now we prove the first statement in Claim 1. Obviously we have $N_q\subset N_{q'}$ if $q'$ is a multiple of $q$. We assume conversely that $N_{q'}\backslash N_q$ is non-empty, then by Lemma \ref{lemma-3-4}, $\dim N_{q'}>1$ and we can find infinite prime closed geodesics on $N_{q'}$. This is a contradiction. So we must have $N_{q'}=N_q$ in this case.

Next we prove the second statement in Claim 1.
Notice that $N_q$ is a closed reversible geodesic,
on which each point $x$ satisfies
$$\psi(x)\neq x,\ldots,\psi^{q-1}(x)\neq x,\psi^q(x)=x.$$
We can easily see that in each direction the length of $N_q$ can be presented as $p\pi/q$ for some positive integer $p$.

When $q'$ is not a multiple of $q$, $N_k\cap N_{q'}=\emptyset$. If $N_{q'}\neq\emptyset$,  $N_{qq'}$ contains both $N_{q}$ and $N_{q'}$. By Lemma \ref{lemma-3-4}, $(N_{qq'},F|_{N_{qq'}})$ is  a Finsler sphere with $\dim N_{qq'}>1$, $K\equiv 1$ and only finite prime closed geodesics. The restriction of $\psi^{qq'}$
to $N_{qq'}$ is the identity map, so there exist infinite prime
closed geodesics, which is a contradiction to Assumption (F).

This ends the proof of Claim 1 as well as Lemma \ref{lemma-3-5} when some $N_q$ is a geodesic.

Now we consider the case that no $N_k$ has a positive
dimension. Using the observation $N_{k'}\cup N_{k''}\subset N_{k'k''}$ and Lemma \ref{lemma-3-4}, we see that the only non-empty
$N_k$'s consist of the same pair of points $x'$ and $x''$
satisfying $d_F(x',x'')=d_F(x'',x')=\pi$.

This ends the proof of Claim 1 when all $N_q$'s are discrete.
\end{proof}

The group that $\psi$ generates in
$I(M,F)$ is not closed. So $I(M,F)$ is
a compact Lie group of positive dimension. We denote
$\mathfrak{g}=\mathrm{Lie}(I(M,F))$, which is the space of
all Killing vector fields of $(M,F)$.

Because of Assumption (F), an isometry of $(M,F)$ preserves each closed geodesic, and a Killing vector field is tangent to each closed geodesic. For any closed geodesic $c=c(t)$ of constant speed, the restriction of a Killing vector field $X$ to $c(t)$ must satisfy
\begin{equation}\label{002}
X(c(t))=\rho_{X,c}\dot{c}(t),\quad\forall t\in\mathbb{R},
\end{equation}
where $\rho_{X,c}$ is a real number which only depends on
the Killing vector field $X$ and the geodesic $c=c(t)$.

The following lemma indicates when $\rho_{X,c}$ vanishes for each closed geodesic $c=c(t)$, the Killing vector field $X$ must be the zero vector field.

\begin{lemma}\label{lemma-3-6}
Let $(M,F)=(S^n,F)$ be a Finsler sphere with $n>1$, $K\equiv 1$
and only finite prime closed geodesics. Then
any Killing vector field of $(M,F)$ which vanishes on
all closed geodesics must be the zero vector field.
\end{lemma}

\begin{proof}
Assume conversely that there exists a nonzero Killing
vector field of $(M,F)$ which vanishes on all closed geodesics. Then all the Killing vector fields satisfying
this requirement generate a nontrivial closed connected subgroup $G'$ of $I_0(M,F)$. We can find a nonzero Killing vector field $X$ in $\mathfrak{g}'=\mathrm{Lie}(G')$, which generates an $S^1$-subgroup. Assume the function
$f(\cdot)=F(X(\cdot))^2$ achieves its maximum at $x\in M$.
Then by Lemma \ref{lemma-2-1}, the integration curve $c(t)$ of $X$
passing $x$ is a closed geodesic. But $X$ does  not vanish
on $c(t)$. This is a contradiction.
\end{proof}

Using Lemma \ref{lemma-3-6}, it is not hard to prove Theorem
\ref{main-thm-0}.

\noindent{\bf Proof of Theorem \ref{main-thm-0}.}
We first prove $I_0(M,F)$ is a torus.

Let $\mathfrak{g}'$ be the semi-simple factor of $I_0(M,F)$. For
each closed geodesic $c(t)$, the number $\rho_{X,c}$ in (\ref{002}) defines
a one dimensional real representation of $\mathfrak{g}=\mathrm{Lie}(I(M,F))$. Its restriction to $\mathfrak{g}'$ is trivial. So each Killing vector field $X$
in $\mathfrak{g}'$ vanishes on all closed geodesics. By Lemma
\ref{lemma-3-6}, $X$ must be zero, i.e. $\mathfrak{g}'=0$ and
$I_0(M,F)$ is a torus.

Next we prove $0<\dim I_0(M,F)\leq [\frac{n+1}{2}]$.

We have already seen $\dim I_0(M,F)>0$. The proof of Lemma \ref{lemma-3-6} shows that we can find a closed geodesic $c(t)$ such that
$\mathfrak{g}_c=\{X|\rho_{X,c}=0\}$ is a codimension one subspace
in $\mathfrak{g}$. The subalgebra $\mathfrak{g}_c$ generates the closed subgroup $G'$ of $I_0(M,F)$ which fixes each point of the geodesic $c(t)$. At $x=c(0)$, the action of $G'$ preserves the tangent subspace
$$\mathbf{V}=\{u|\langle u,\dot{c}(0)\rangle_{\dot{c}(0)}^F=0\}\subset T_xM$$
as well as the inner product $g_x^F$. This action is effective,
 making $G'$ to be a torus subgroup of $SO(n-1)$. So we have
$$\dim I_0(M,F)=\dim G' +1\leq\mathrm{rank}SO(n-1)+1=[\frac{n+1}{2}].$$

This ends the proof of Theorem \ref{main-thm-0}.
{\ \rule{0.5em}{0.5em}}

As a contrast with Lemma \ref{lemma-3-6}, we also have
the following lemma.
\begin{lemma}\label{lemma-3-7}
Let $(M,F)=(S^n,F)$ be a Finsler sphere with $n>1$, $K\equiv 1$
and only finite prime closed geodesics, $N$ the fixed point set of $I_0(M,F)$, or equivalently the common zero set of all Killing vector fields. Assume $N$ is not empty, then it must be the fixed point set of some $\psi^k$, which is listed in Lemma \ref{lemma-3-5}, i.e. a two-points $\psi$-orbit, or a reversible closed geodesic.
\end{lemma}

\begin{proof}
The isometry group $I(M,F)$ is a compact Lie group, so $I(M,F)/I_0(M,F)$ is finite. If the fixed point set $N$ of
$I_0(M,F)$ is not empty, then it is contained in the fixed point set $N_k$ of $\psi^k$, where $k=|I(M,F)/I_0(M,F)|$. By Lemma \ref{lemma-3-5}, $N_k$ is either a two-points $\psi$-orbit
or a reversible geodesic. When $N_k$ is a geodesic, it is
obvious to see that all Killing vector field vanish on $N_k$, i.e. $N_k=N$. When $N_k$ consists of two points, we also have
$N_k=N$ because $N$ consists of $\psi$-orbits.
\end{proof}

At the end of this section, we remark that the prime closed geodesics of $(M,F)$ are imbedded curves which do not intersect each other when they are geometrically distinct.

\begin{lemma}\label{lemma-3-8}
Let $(M,F)=(S^n,F)$ with $n>1$ be a Finsler sphere with $K\equiv 1$ and only finite prime closed geodesics.
Then any prime closed geodesic has no self intersections. If two prime closed geodesics are geometrically distinct, then they do not
intersect.
\end{lemma}

\begin{proof} To prove the first statement, we
assume conversely
 that the closed geodesic $\gamma$ has self intersections.
 Then the action of $I_0(M,F)$ on $\gamma$ is trivial. By Lemma \ref{lemma-3-7}, $\gamma$
 is a connected totally geodesic sub-manifold, which can not have self intersection. This is a contradiction.

 To prove the second statement, we assume conversely
 two geometrically distinct closed geodesics $\gamma_1$ and $\gamma_2$ intersect. Then the action of $I_0(M,F)$ on
 $\gamma_1\cup\gamma_2$ is trivial. It implies that the fixed point set of $I_0(M,F)$ is a Finsler sphere $N$ with $K\equiv 1$ and $\dim N>1$. This is a contradiction with Lemma \ref{lemma-3-7}.
\end{proof}

\section{Proof of Theorem \ref{main-thm-1}}
Before proving Theorem \ref{main-thm-1}, we observe that
a weaker statement follows Lemma \ref{lemma-3-6} easily.

\begin{lemma}\label{lemma-4-1}
Let $(M,F)=(S^n,F)$ be a Finsler sphere with $n>1$, $K\equiv 1$ and only finite prime closed geodesics.
Then the number of geometrically distinct closed geodesics on $M$ is at least $\dim I(M,F)$.
\end{lemma}

\begin{proof}
Assume conversely that the number of geometrically distinct closed geodesics on $M$ is less than $\dim I(M,F)$. Then there exists
a nonzero Killing vector field $X$, which restriction to each
geodesic is zero. This is a contradiction to Lemma \ref{lemma-3-6}.
\end{proof}

Now we use induction to prove Theorem \ref{main-thm-1}.

When $n=2$, it has been proved in \cite{BFIMZ2017}.

Now we assume Theorem \ref{main-thm-1} is valid for $k<n$ and prove the statement when $k=n$.

Firstly, we prove the claim

 {\bf Claim 2}. If $\dim I_0(M,F)=1$, then there exists
a closed reversible geodesic.

Let $X$ be a nonzero Killing vector field. We first prove Claim 2 when $X$ has a non-empty zero point set.

By Lemma \ref{lemma-3-7}, if the zero point set $N$ of $X$ is non-empty, it is either a closed reversible geodesic or a two-points $\psi$-orbit $\{x',x''\}$.
When $\dim N$ is a closed reversible geodesic, the proof for Claim 2 is done. Now we assume $N=\{x',x''\}$ is a two-points $\psi$-orbit. By Lemma \ref{lemma-2-1}, we can find a closed geodesic $\gamma$, passing the maximum point $x$ for the function $f(\cdot)=F(X(\cdot))^2$. So the restriction of $X$ to $\gamma$ is not zero.  We can find a suitable
isometry $\phi=\psi\exp tX$ which fixes each point of $\gamma$.
The isometry $\phi$ is not the identity map because it exchanges $x'$ and $x''$. By Lemma \ref{lemma-3-4}, the fixed point set $N$ of $\phi$ is either a reversible closed geodesic, or a Finsler
sphere with $1<\dim N<n$, $K\equiv 1$ and only finite prime closed geodesics. For the second case, we can use
the inductive assumption to find a reversible closed geodesic of $N$, as well as of $M$.

This ends the proof of Claim 2 when $X$ has a non-empty zero set.

Next we prove the Claim 2 assuming $X$ is nonvanishing everywhere. Let $c_i=c_i(t)$ with $t\in[0,t_i]$ and $c_i(0)=c_i(t_i)$, $1\leq i\leq m$, be all the prime closed geodesics such that $\dot{c}_i$ has the same direction as $X$. We may arrange them such that $t_1\leq\cdots\leq t_m$,
and we can parametrize each $c_i(t)$, such that $X(c_i(t))=\dot{c}_i(t)$ for all $t$.

If $t_i$'s are not all the same, the fixed point set $N$ of $\phi=\exp t_1 X$ contains the geodesic $c_1$ but not $c_m$. We just switch to $(N,F|_N)$, which is either a reversible geodesic in $M$ or a Finsler sphere with $1<\dim N<n$, $K\equiv 1$
and only finite prime closed geodesics. By the inductive assumption, we can find a reversible closed geodesic
for $N$, as well for $M$.

If the antipodal map $\psi$ is not contained in $I_0(M,F)$. For each closed geodesic $\gamma$, we can find a suitable isometry
$\phi=\psi\exp tX$ such that $\phi$ fixes each point of $\gamma$. Notice that $\phi$ is not the identity map because $\psi\notin I_0(M,F)$. Then we can switch to the fixed point set of
$\phi$. By the inductive assumption and similar argument as above, we can find a reversible closed geodesic.

Now we assume that all $t_i$'s are the same, and at the same time $\psi$ belongs to $I_0(M,F)$, i.e. $\psi=\exp tX$ for some suitable $t$. Let $t'_i\in(0,t_i)$ be the number such that $\psi(c_i(0))=c(t'_i)$. Then we can arrange these geodesics such that $t'_1\leq \cdots\leq t'_m$.

The case that all $t'_i$'s are the same number can not happen.
Assume conversely that all $t'_i$'s are the same. We consider
the function $f(\cdot)=F(X(\cdot))^2$ on $M$. Because $X$ is
nonvanishing everywhere, $f$ is a positive smooth function on $M$. By Lemma \ref{lemma-2-1}, the critical point set of $f(\cdot)$ is just the union of all the geodesics $c_i$'s.
On the other hand, our assumption implies  $f(\cdot)$ takes
the same positive value on these geodesics. So $f(\cdot)$ is
a constant function on $M$, i.e. $X$ is Killing vector field of constant length. Using Lemma \ref{lemma-2-1} again, each integration curve of $X$
is a closed geodesic. This is a contradiction.

We re-order these geodesics such that $t'_1\leq \cdots\leq t'_m$.
Then the isometry $\phi=\psi\exp (t_m-t'_m)X$ fixes the points on $\gamma_m$ but not those on $\gamma_1$. Using similar arugment and the
inductive assumption as above, we can find a reversible closed geodesic in the fixed point set of $\phi$. This ends the proof of the Claim 2.

Secondly, we assume $m=\dim I(M,F)>1$ and prove

{\bf Claim 3.} The number of geometrically distinct prime closed geodesics is at least $m-1$ when $m=\dim I(M,F)>1$.

Let $\{\gamma_1, \gamma_2, \ldots, \gamma_{m'}\}$ be
the set of all geometrically distinct closed geodesics.
By Lemma \ref{lemma-4-1}, $m'\geq m$. For each $i$, $1\leq i\leq m'$, denote $\mathfrak{g}_i$ the subspace of Killing vector fields which vanishes on $\gamma_i$. By Lemma \ref{lemma-3-6}, $\cap_{1\leq i\leq m'}\mathfrak{g}_i=0$. Each $\mathfrak{g}_i$ is either $\mathfrak{g}$ or a codimension one subspace of $\mathfrak{g}$. If some $\mathfrak{g}_i=\mathfrak{g}$,
We have already observed that that $\gamma_i$ is the fixed point set of some $\psi^k$, which is unique. By Lemma \ref{lemma-3-6}, we can re-arrange all the
geodesics, such that $\mathfrak{g}_i$'s for $1\leq i\leq m$ are
$(m-1)$-dimensional and their intersection is zero. We can find
a nonzero Killing vector field
$X\in\cap_{1\leq i\leq m-1}\mathfrak{g}_i$.

By Lemma \ref{lemma-3-4}, the zero set $N$ of $X$
is either a reversible closed geodesic, or a Finsler sphere with $1<\dim N<n$, $K\equiv 1$ and only finitely many prime closed geodesics. The first case only happens when $m=2$, the proof of Claim 3 is done because we have already found $m-1=1$ reversible closed geodesic. For the second case, we observe that
the restriction to $N$ is injective map from $\mathfrak{g}_m$ to $\mathrm{Lie}(I(N,F|_N))$, i.e. $\dim I(N,F|_N)\geq m-1$. So by
the inductive assumption, the number of geometrically distinct reversible prime closed geodesics for $N$, as well as for $M$, is at least $m-1$. This ends the proof of Claim 3.

Finally, we finish the proof by induction. We re-arrange all the geometrically distinct geodesics such that $\gamma_i$'s
with $1\leq i\leq m-1$ are reversible. Notice that any $\mathfrak{g}_k$
is contained in $\cup_{1\leq i\leq m-1}\mathfrak{g}_i$ iff
$\mathfrak{g}_k=\mathfrak{g}_i$ for some $i$, $1\leq i\leq m-1$.
Comparing the facts that $\dim\cap_{1\leq i\leq m-1}\mathfrak{g}_i> 0$ and $\cap_{1\leq i\leq m'}\mathfrak{g}_i=\emptyset$, we may assume $\mathfrak{g}_m$
is not contained in $\cup_{1\leq i\leq m-1}\mathfrak{g}_i$.
Choose a Killing vector field
$X\in\mathfrak{g}_m\backslash\cup_{1\leq i\leq m-1}\mathfrak{g}_i$, then by Lemma \ref{lemma-3-5} and the induction argument, the zero set of $X$ contains a prime closed geodesic which is geometrically different with those $\gamma_i$'s with
$1\leq i\leq m-1$. So there exists at least $m$ geometrically
distinct reversible closed geodesics.

This ends the induction and proves Theorem \ref{main-thm-1}.

\section{Proof of Theorem \ref{main-thm-2}}

In this section, we assume the dimension of $\mathfrak{g}= \mathrm{Lie}(I(M,F))$ is $[\frac{n+1}{2}]$, and
prove Theorem \ref{main-thm-2}.

\subsection{The case that $\dim M$ is odd}

We consider the case that $\dim M=n=2n'+1$ is odd, in which $n'>0$. In this case $\dim I(M,F)=n'+1$.

Firstly, we prove the following claim.

 {\bf Claim 4}.
All closed geodesics are reversible.

Let $c(t)$ be any closed geodesic and
$\mathfrak{g}'$ the space of all Killing vector fields which
 vanish on $c(t)$. Then $n'\leq\dim \mathfrak{g}'\leq n'+1$. It generates a torus subgroup $G'$ in $I_0(M,F)$ which action preserves the $g_y^F$-orthogonal complement $\mathbf{V}$ of $y=\dot{c}(0)$ in the tangent space $T_xM$ at $x=c(0)$.
This action defines an injective endomorphism from $G'$ to
$SO(2n')$, which image must be a maximal torus for dimension reason. At the same time, we get $\dim \mathfrak{g}'=n'$. There is an isometry $\phi\in G'$ which acts as $-\mathrm{id}$ on $\mathbf{V}$. So the geodesic $c(t)$ is a connected component of the fixed point set of $\phi$, which must be a reversible geodesic. This ends the proof of Claim 4.

Secondly, we consider the function $f(x)=F(X(x))^2$ on $M$ for
a Killing vector field $X$ which generates a closed subgroup $T=S^1$.
We claim

 {\bf Claim 5}. For a generic Killing vector field $X$, $f(\cdot)=F(X(\cdot))^2$ is a positive smooth function on $M$.

To prove this claim, we only need to observe that a generic $X$
is nonvanishing everywhere, then the smoothness follows easily.

Let $\gamma_i$ with $1\leq i\leq m'$ be all the geometrically distinct closed geodesics, and $\mathfrak{g}_i\subset\mathfrak{g}$ be the space of Killing vector fields which vanish on $\gamma_i$
respectively. The argument proving Claim 4 shows
$\dim\mathfrak{g}_i=\dim\mathfrak{g}-1$, so when the Killing vector field $X$ is chosen from
$\mathfrak{g}\backslash \cup_{1\leq i\leq m'}\mathfrak{g}_i$,
it does not vanish on any closed geodesic.
The zero set of $X$ can not be a two-points $\psi$-orbit (case (1) of Lemma \ref{lemma-3-4}), because it must have an even codimension in $M$. It can not have positive dimension either,
because otherwise $X$ vanishes on some closed geodesics. So the
cases (2) and (3) of Lemma \ref{lemma-3-4} do not happen.
To summarize, we see a generic Killing vector field $X$ is nonvanishing everywhere, which proves Claim 5.

Thirdly, we choose a generic Killing vector field $X\in\mathfrak{g}\backslash
\cup_{1\leq i\leq m'}\mathfrak{g}_i$ which generates a closed subgroup $T=S^1$ in $I_0(M,F)$. We will discuss the critical submanifolds of $f(\cdot)=F(X(\cdot))^2$, their nullities and indices.

The function $f(\cdot)$ is $T$-invariant. So its critical sub-manifold consists of $T$-orbits. Notice that by Claim 4,
all closed geodesics are reversible. So
by Lemma \ref{lemma-2-1}
and Assumption (F), we have a one-to-one correspondence between connected critical sub-manifold of $f(\cdot)$ and all the geometrically distinct closed geodesic.

Now we calculate the nullity and index for $f(x)$ at each closed geodesic.

Let $c(t)$ with $t\in[0,t_0]$ and $c(0)=c(t_0)$ be a prime closed geodesic with unit speed. Without loss of generalities, we may normalize $X$ such that $X(c(t))=\dot{c}(t)$.

To avoid unnecessary complexity, in the discussion below
we will replace the metric $F$ by its osculation metric $g^F_X=\langle\cdot,\cdot\rangle$.
Because $X$ is nonvanishing everywhere, the Riemannian metric $g^F_X$ is globally defined. Denote $R(\cdot,\cdot)$ and
$\nabla$ the curvature operator and the Lev-Civita connection for $(M,g^F_X)$.

The action of $I_0(M,F)$ on $M$ preserves $F$ and $X$, so it makes $I_0(M,F)$
a subgroup of $I_0(M,g_X^F)$.
Notice that $$f(\cdot)=F(X(\cdot))^2=\langle X(\cdot),X(\cdot)\rangle.$$ So $c(t)$ is geodesic for $(M,g_X^F)$ because it is a critical sub-manifold for
$f(\cdot)$. For any $x=c(t)$ on this geodesic, and any tangent plane $\mathbf{P}$ containing $y=X(x)=\dot{c}(t)$, we have
$K^{g^F_X}(x,\mathbf{P})=K^F(x,y,\mathbf{P})=1$ by Theorem \ref{thm-geodesic-field}.

Denote $x=c(0)=c(t_0)$, $\mathbf{V}$ the orthogonal complement of $y=\dot{c}(0)$ in $T_xM$, and $G'$ the closed subgroup of
$I_0(M,F)$ fixing $x$. We can find an orthonormal basis $\{e^i,1\leq i\leq 2n'\}$ for $\mathbf{V}$, such that the $G'$-actions on $\mathbf{V}$ are of the form
\begin{equation}\label{005}
\mathrm{diag}(R(t_1),\ldots,R(t_{n'})),
\end{equation}
in which $$R(t_i)=\left(
              \begin{array}{cc}
                \cos t_i & \sin t_i \\
                -\sin t_i & \cos t_i \\
              \end{array}
            \right).$$

The parallel movement along $c(t)$ from $t=0$ to $t=t_0$ defines an orientation preserving orthogonal linear map $L$ on $\mathbf{V}$, which commutes with all the linear maps in (\ref{005}). So we can present $L$ as $$L=\mathrm{diag}(R(\alpha_1),\ldots,R(\alpha_{n'})).$$
Similarly, the tangent map of the action of any $\varphi\in I_0(M,F)$ at $x$, restricted to $\mathbf{V}\subset T_xM$, is of the form
$$P_{s_{0}}\circ
\mathrm{diag}(R(s_1),\ldots,R({s_{n'}})),$$
where $P_{s_0}$ is the parallel movement along $c(t)$ from $t=0$ to $t=s_0$.

In particular when $\varphi=\exp tX$ in which the Killing vector field $X$ satisfies $X(c(t))=\dot{c}(t)$, we have
$$s_0=t,\mbox{ and }s_i=t\theta_i,\quad\forall 1\leq i\leq n'.$$

For each nonzero vector $v\in \mathbf{V}$, we can find constant speed geodesic $\gamma'$ initiating from $x$ in the direction of $v$. The $T$-action on $\gamma'$ defines a local variation $c(s,t)$ of $c(t)$, i.e. when $s$ and $t$ are sufficiently close to 0,
$c(s,t)\in M$ is smoothly well defined, $c(0,t)=c(t)$,
$c_t(s,t)=X$ and $c_s(0,0)=v$ and $c_s(s,0)$ is the geodesic $\gamma'$.
Because $X$ is a Killing vector field, each integration curve of $V=c_s(s,t)$ is a geodesic, restricted to which $X$ is a Jacobi field. Also notice $[X,V]=0$, so we have
\begin{eqnarray}\label{003}
\frac{1}{2}\frac{\partial^2}{\partial s^2}\langle X,X\rangle
&=&V\langle\nabla_VX,X\rangle\nonumber\\
&=&\langle\nabla_V\nabla_VX,X\rangle+
\langle\nabla_VX,\nabla_VX\rangle\nonumber\\
&=&-\langle R(X,V)V,X\rangle+\langle\nabla_XV,\nabla_XV\rangle.
\end{eqnarray}

When $s=0$, the first term at the right side of (\ref{003}) is $$-\langle R(X,V)V,X\rangle=-\langle V,V\rangle.$$

To calculate the second term when $s=0$, we
use the orthonormal basis $\{e^i,1\leq i\leq 2n'\}$ given above,
$v=v_ie^i$ can be presented as a column vector and
$$V(c(t))=c_s(s,t)|_{s=0}
=T_t(\mathrm{diag}(R(t\theta_1),\ldots,R(t\theta_{n'}))v).$$
So when $s=t=0$ we have
$$\nabla_XV=\mathrm{diag}(\theta_1 J,\ldots,\theta_{n'}J)v,$$
in which $$J=\left(
              \begin{array}{cc}
                0 & 1 \\
                -1 & 0 \\
              \end{array}
            \right).$$
To summarize, we get
$$\frac{1}{2}
\frac{\partial^2}{\partial s^2}\langle X,X\rangle|_{s=t=0}
=\sum_{i=1}^{n'}(\theta_{i}^2-1)(v_{2i-1}^2+v_{2i}^2),
$$
with $v=v_ie^i$.
Now it is easy to see
the nullity and index of the critical submanifold $c=c(t)$ for $f(\cdot)=\langle X(\cdot),X(\cdot)\rangle$
is two times the number of $\theta_i$'s satisfying $|\theta_i|=1$ and
$|\theta_i|<1$ respectively. In particular the index is an even number.

Using the parallel movement, each pair $e^{2i-1}$ and $e^{2i}$ generate a two dimension trivial bundle
on the geodesic $c(t)$. The negative bundle on the geodesic $c(t)$ is generated by those $e^{2i-1}$ and $e^{2i}$ with $|\theta_i|<1$, i.e. a sum of two dimensional trivial bundles, so it is a trivial bundle as well.

In above discussion, we normalized $X$ such that $X(c(t))=\dot{c}(t)$ has the unit length. Back to a general generic $X$, each condition $|\theta_i|=1$ corresponds to two
co-dimensional one subspaces in $\mathfrak{g}=\mathrm{Lie}(I(M,F))$. From the complement of
the union of all these subspaces for all $1\leq i\leq n'$ and
all prime closed geodesics, we can still find $X$ from the dense subset of $\mathfrak{g}$ of Killing vectors which generate an $S^1$-subgroup.

All above discussion can be summarized as the following lemma.
\begin{lemma}\label{lemma-5-1}
Let $(M,F)=(S^n,F)$ be an odd dimensional Finsler sphere with $n\geq 3$, $K\equiv 1$, only finite prime closed geodesics, and $\dim I(M,F)=\frac{n+1}{2}$. Then there exists a Killing vector field
$X$ satisfying the following conditions:
\begin{description}
\item{\rm (1)} The Killing vector field $X$ is nonvanishing everywhere and generates
a subgroup $T=S^1$ in $I(M,F)$.
\item{\rm (2)} The function $f(\cdot)=F(X(\cdot))^2$ is a positive smooth function on $M$ and the critical sub-manifolds are all the closed geodesics.
\item{\rm (3)} The function $f(\cdot)$ is a $T$-equivariant
Morse function, i.e. for each prime closed geodesic viewed as a critical sub-manifold of $f(\cdot)$, the nullity is zero.
\item{\rm (4)} The index for each closed geodesic is even, and the negative bundle is a trivial.
\end{description}
\end{lemma}

Finally, we fix a Killing vector field $X$ as indicated in
Lemma \ref{lemma-5-1}, and use the equivariant Morse theory to
finish the proof of Theorem \ref{main-thm-2} when $\dim M=n$ is odd.

Let $T=S^1$ be the subgroup that $X$ generates in $I(M,F)$.
For the $T$-action on $M=S^{n}$, we have the following $T$-equivariant cohomology groups with $\mathbb{R}$-coefficient.

Denote $BT$ be the classification space for principal $T$-bundles and $ET$ the universal principal $T$-bundle over $BT$. Define $M_T=M\times ET/T$ in which $T$ acts diagonally.
Then $H^*(M_T;\mathbb{R})$ is the $T$-equivariant cohomology
ring with $\mathbb{R}$-coefficient, denoted as
$H^*_T(M;\mathbb{R})$. The formal series
$$P_T(M;\mathbb{R},t)=\sum_{i=0}^\infty t^i\dim H^i_T(M;\mathbb{R})$$
is called the $T$-equivariant Poincare polynomial.
From the Serre-Leray spectral sequence, we see it is well defined, i.e.
$\dim H^i_T(M;\mathbb{R})<+\infty$ for each $i$.

Let $\gamma_i$ with $1\leq i\leq m$ be all the geometrically
distinct closed geodesics. Assume the isotropy group for the $T$-actions on each $\gamma_i$ is $\mathbb{Z}_{k_i}$. Each $\gamma_i$ is a non-degenerate critical submanifold for $f(\cdot)=F(X(\cdot))^2$ which index $\lambda_i$ is even. These data about critical sub-manifolds bring us another polynomial, i.e.
\begin{eqnarray}\label{004}
\mathcal{M}_G(M;f,\mathbb{R},t)&=&\sum_{i=1}^{m}
t^{\lambda_i}P_T(\gamma_i;\mathbb{R},t)\nonumber\\
&=&\sum_{i=1}^{m}
t^{\lambda_i}\sum_{j=0}^\infty t^j\dim H^j(B\mathbb{Z}_{k_i};\mathbb{R})\nonumber\\
&=&\sum_{i=1}^m t^{\lambda_i}.
\end{eqnarray}

Notice that in (\ref{004}), we need to use a $T$-equivariant
Thom isomorphism which requires an orientability of the negative bundle. This can be guaranteed by Lemma \ref{lemma-5-1} (4), i.e. the negative bundle is trivial, and the connectedness of $T=S^1$.

The $T$-equivariant Morse theory \cite{Bo1954}\cite{Bo1982}\cite{Hi1984} tells us
\begin{equation}
\sum_{i=1}^m t^{\lambda_i}=P_T(M;\mathbb{R},t)+(1+t)Q(t),\label{001}
\end{equation}
in which $Q(t)$ is a formal power series with non-negative integer coefficients. Because the left side of (\ref{001})
is a polynomial of $t^2$, so does $P_T(M;\mathbb{R},t)$ and
$Q(t)\equiv0$.
The number $m$ of all geometrically distinct closed geodesics
coincides with the $T$-equivariant Euler number
$\chi_T(M;\mathbb{R})=P_T(M;\mathbb{R},-1)$.

The $T$-equivariant Euler number $\chi_T(M;\mathbb{R})$ can be calculated as following. Because $M_T$ is a fiber bundle over $BT$ with fiber $M$, so we have the
Leray-Serre spectral sequence \cite{BT1982} for $M_T$, i.e.
\begin{eqnarray*}
E_2^{p,q}&=&H^p(BT;H^q(S^n,\mathbb{R}))\\
&=&H^p(BT;\mathbb{R})\times H^q(S^n;\mathbb{R})
\Rightarrow H^{r}(M_T;\mathbb{R})=H^r_T(M;\mathbb{R}).
\end{eqnarray*}
All the nonzero $E_2^{*,*}$-terms are $E_2^{p,q}=\mathbb{R}$ when $p$ is a non-negative even number and $q=0$ or $n$, so we have
\begin{eqnarray*}
& &E_2^{*,*}\cong E_3^{*,*}\cong\cdots\cong E_{n+1}^{*,*},\mbox{ and}\\
& &E_{n+2}^{*,*}\cong E_{n+2}^{*,*}\cong\cdots\cong
E_\infty^{*,*}.
\end{eqnarray*}
The differential map
$$d_{n+1}:\mathbb{R}= E_{n}^{2k,n}\rightarrow E_{n}^{2k+n+1,0}=\mathbb{R}$$
must be an isomorphism for each $k\geq 0$, because otherwise
$$H_T^k(M_T;\mathbb{R})\cong\oplus_{p+q=k}E_{\infty}^{p,q}
=\oplus_{p+q=k}E_{n+2}^{p,q}\neq 0$$ for some odd $k$,
which is a contradiction to the fact that $P_T(M;\mathbb{R},t)$
is a polynomial of $t^2$.

To summarize, we get $$H^{2k}_T(M;\mathbb{R})\cong E_{\infty}^{2k,0}=\mathbb{R}$$
when $0\leq k\leq[n/2]$, and all other $E_{\infty}^{p,q}=0$ and
$H_T^r(M;\mathbb{R})$ vanish. So
$$m=\chi_T(M;\mathbb{R})=\frac{n+1}{2},$$
which ends the proof of Theorem \ref{main-thm-2} when $n$ is odd.

\subsection{The case that $\dim M$ is even}

Now we prove Theorem \ref{main-thm-2} with the assumption $\dim M=n=2n'$ is even. The case $n'=1$ has been proved by Theorem 2 of \cite{BFIMZ2017}. So we can assume $n'>1$. In this case $\dim I(M,F)=n'$. Notice that the
 antipodal map $\psi$ reverses the orientation.

As before, we denote all the geometrically distinct prime closed geodesics as $\gamma_i$, $1\leq i\leq m$,
and $\mathfrak{g}_i$ the space of all Killing vector fields
which vanishes on $\gamma_i$. Notice that $m\geq n'$ by Lemma \ref{lemma-4-1}. By Lemma \ref{lemma-3-6}, we can re-order those geodesics,
such that $\cap_{1\leq i\leq n'}\mathfrak{g}_i=0$. Then the action of $I_0(M,F)$ on each geodesic $\gamma_i$ with $1\leq i\leq n'$ induces a diagonal action on
$$\gamma_1\times\cdots\times\gamma_{n'}\subset
M\times\cdots\times M,$$
which is transitive. So we can find an isometry of the form
$\phi=\psi\phi_0$ with $\phi_0\in I_0(M,F)$ such that $\phi$
acts trivially on each $\gamma_i$ for $1\leq i\leq n'$. Because $\phi$ reverses the orientation, its fixed point set $N$ is an odd
dimensional Finsler sphere with $\dim N\geq 3$, $K\equiv 1$ and only finite prime closed geodesics.

Because $\phi$ commutes with $I_0(M,F)$, the action $I_0(M,F)$ preserves the fixed point set $N$ of $\phi$.
The restriction to $N$ is an
injective linear map from $\mathrm{Lie}(I(M,F))$ to
$\mathrm{Lie}(I(N,F|_N))$, so
$\dim I(N,F|_N)\geq n'$.
On the other hand,
$\dim N\leq 2n'-1$.
By Theorem \ref{main-thm-0}, $\dim I(N,F|_N)\leq [\frac{\dim N+1}{2}]\leq n'$.
To summarize, we get
$$\dim N=2n'-1\mbox{ and }\dim I(N,F|_{F|_N})=n'.$$
 By the proof in Subsection 5.1, we see $\gamma_i$'s with $1\leq i\leq n'$ are exactly all the geometric distinct
closed geodesics on $N$, and they are all reversible. We just need to prove that
there exist no more closed geodesics on $M$.

Assume conversely that there exists a closed geodesic
$\gamma$ which is
geometrically different with all $\gamma_i$ with $1\leq i\leq n'$. By Lemma \ref{lemma-3-3}, $\gamma$ intersects with $N$.
The action of $I_0(M,F)$ preserves both $\gamma$ and $N$, so $\gamma$ is contained in the fixed point of $I_0(M,F)$.
By Lemma \ref{lemma-3-7},
$\gamma$ is the fixed point set of some
$\psi^k$.

If $N\cap\gamma$ contains more than two points, by Lemma \ref{lemma-3-5}, the fixed point set of $\psi^k$ in $N$ is another closed geodesic $\gamma'$. The two geometrically distinct closed geodesics
$\gamma$ and $\gamma'$ intersect. This is a contradiction
to Lemma \ref{lemma-3-8}.

So $N\cap\gamma$ is a two-points $\psi$-orbit. Using Lemma
\ref{lemma-3-5} again, we see $\gamma$ is the fixed point set
of $\psi^2$ in $M$. But $\dim M$ is even and $\psi^2$ is an orientation preserving isometry, so the fixed point set of $\psi^2$ must have an even dimension.
This is a contradiction which ends the proof of Theorem \ref{main-thm-2}.

\end{document}